\documentclass{amsart}

\usepackage[latin1]{inputenc}
\usepackage[english]{babel}
\usepackage{indentfirst}
\usepackage{amssymb}
\usepackage{amsthm}
\usepackage{xcolor}
\usepackage[all]{xy}
\usepackage[mathscr]{eucal}

\newcommand{\impli}{\Rightarrow}

\newcommand{\N}{\mathbb{N}}

\newcommand{\sub}{\subseteq}
\def\epsilon{\varepsilon}

\newtheorem{theorem}{Theorem}[section]
\newtheorem{proposition}[theorem]{Proposition}
\newtheorem{corollary}[theorem]{Corollary}
\newtheorem{lemma}[theorem]{Lemma}

\newtheorem{definition}[theorem]{Definition}
\newtheorem{example}[theorem]{Example}
\newtheorem{remark}[theorem]{Remark}

\numberwithin{equation}{section}

\author{S. Okada}
\address{112 Marcorni Crescent, Kambah, ACT 2902, Australia} \email{susbobby@grapevine.com.au}

\author{J. Rodr\'{i}guez}
\address{Dpto. de Ingenier\'{i}a y Tecnolog\'{i}a de Computadores\\Facultad de Inform\'{a}tica\\
Universidad de Murcia\\ 30100 Espinardo (Murcia)\\ Spain} \email{joserr@um.es}

\author{E.A. S\'{a}nchez-P\'{e}rez}
\address{Instituto Universitario de Matem\'{a}tica Pura y Aplicada\\ Universitat Polit\`{e}cnica de Val\`{e}ncia\\
Camino de Vera s/n\\ 46022 Valencia\\ Spain} \email{easancpe@mat.upv.es}



\thanks{The research of J. Rodr\'{i}guez was partially supported by grants PID2021-122126NB-C32 
(funded by MCIN/AEI/10.13039/501100011033 and ``ERDF A way of making Europe'') and 
21955/PI/22 (funded by {\em Fundaci\'on S\'eneca - ACyT Regi\'{o}n de Murcia}). The research of E.A. S\'{a}nchez-P\'{e}rez was partially supported 
by grant PID2020-112759GB-I00 funded by MCIN/AEI/10.13039/501100011033.}

\title{On vector measures with values in $\ell_\infty$}

\begin{document}

\begin{abstract}
We study some aspects of countably additive vector measures with values in~$\ell_\infty$ and the Banach lattices of 
real-valued functions that are integrable with respect to such a vector measure. On the one hand, we prove that if $W \sub \ell_\infty^*$
is a total set not containing sets equivalent to the canonical basis of~$\ell_1(\mathfrak{c})$, then there is a 
non-countably additive $\ell_\infty$-valued map $\nu$ defined on a $\sigma$-algebra such that 
the composition $x^* \circ \nu$ is countably additive for every $x^*\in W$.
On the other hand, we show that a Banach lattice $E$ is separable whenever it admits a countable positively norming set
and both $E$ and~$E^*$ are order continuous. As a consequence, if $\nu$ is a countably additive vector measure 
defined on a $\sigma$-algebra and taking values in a separable Banach space, then the space $L_1(\nu)$
is separable whenever $L_1(\nu)^*$ is order continuous.
\end{abstract}

\maketitle

\section{Introduction}

Given a $\sigma$-algebra~$\Sigma$ and a Banach space~$X$, we denote by ${\rm ca}(\Sigma,X)$ the set of all countably additive $X$-valued measures
defined on~$\Sigma$; when $X$ is the real field, this set is simply denoted by ${\rm ca}(\Sigma)$. 
The topological dual of~$X$ is denoted by~$X^*$.
The Orlicz-Pettis theorem states that a map $\nu: \Sigma \to X$ belongs to ${\rm ca}(\Sigma,X)$ if (and only if)
the composition $x^*\circ \nu$ belongs to ${\rm ca}(\Sigma)$ for every $x^*\in X^*$ (see, e.g., \cite[p.~22, Corollary~4]{die-uhl-J}). 
It is natural to wonder whether testing on a ``big'' subset, instead of all~$X^*$, is enough for countable additivity.
For instance, one might consider a {\em total} subset of~$X^*$, that is, a set $W \sub X^*$ satisfying $\bigcap_{x^*\in W}\ker x^*=\{0\}$. In general, this does not work. 
A typical example is given by the map $\nu: \mathcal{P}(\N)\to \ell_\infty$ defined on the power set of~$\N$ 
by $\nu(A):=\chi_A$ (the characteristic function of~$A$) for all $A\sub \N$. Indeed, 
$\nu$ is not countably additive, while the composition $\pi_n\circ \nu:  \mathcal{P}(\N) \to \mathbb R$ is countably additive for every $n\in \N$, where $\pi_n\in \ell_\infty^*$
is the $n$th-coordinate functional given by $\pi_n(x):=x(n)$ for all $x\in \ell_\infty$. 

Thanks to the injectivity of~$\ell_\infty$, such an example can be easily carried over to any Banach space
containing subspaces isomorphic to~$\ell_\infty$. Actually, the existence of such subspaces is the only obstacle, as  
the following result of Diestel and Faires~\cite{die-fai} (cf. \cite[p.~23, Corollary~7]{die-uhl-J}) shows:

\begin{theorem}[Diestel-Faires]\label{theorem:DF}
Let $X$ be a Banach space. 
\begin{enumerate}
\item[(i)] Suppose $X$ contains a subspace isomorphic to~$\ell_\infty$. Then there exist a total set $W \sub X^*$ and a map $\nu:\mathcal{P}(\N)\to X$
such that $x^*\circ \nu\in {\rm ca}(\mathcal{P}(\N))$ for every $x^*\in W$ and $\nu\not\in {\rm ca}(\mathcal{P}(\N),X)$.
\item[(ii)] Suppose $X$ does not contain subspaces isomorphic to~$\ell_\infty$. Let $W \sub X^*$ be a total set. If $\nu:\Sigma\to X$
is a map defined on a $\sigma$-algebra~$\Sigma$ such that $x^*\circ \nu\in {\rm ca}(\Sigma)$ for every $x^*\in W$, then 
$\nu\in {\rm ca}(\Sigma,X)$.
\end{enumerate}
\end{theorem}

The following concept (going back to \cite[Appendice~II]{tho}) arises naturally:

\begin{definition}\label{definition:OT}
Let $X$ be a Banach space. A set $W\sub X^*$ is said to have the {\em Orlicz-Thomas (OT) property} if, for every 
map $\nu:\Sigma \to X$ defined on a $\sigma$-algebra~$\Sigma$, we have
$\nu \in {\rm ca}(\Sigma,X)$ whenever $x^*\circ \nu \in {\rm ca}(\Sigma)$ for all $x^*\in W$.
\end{definition}

Any set having the OT property is total (Proposition~\ref{proposition:OTimpliestotal}). 
With this language, the Diestel-Faires Theorem~\ref{theorem:DF} says that every total subset of the dual of a Banach space~$X$
has the OT property if and only if $X$ contains no subspace isomorphic to~$\ell_\infty$. Without additional assumptions on~$X$,
it is easy to check that any norm-dense subset of~$X^*$ has the OT property as explicitly given in \cite[Lemma~3.1]{nyg-rod}; 
alternatively, we can apply the Vitali-Hahn-Saks-Nikod\'{y}m theorem (see, e.g., \cite[p.~24, Corollary~10]{die-uhl-J}) together with the Orlicz-Pettis theorem.
Given any total set $W \sub X^*$, a classical result
of Dieudonn\'{e} and Grothendieck (see, e.g., \cite[p.~16, Corollary~3]{die-uhl-J}) states that an $X$-valued map defined on a~$\sigma$-algebra
is bounded and finitely additive if (and only if) $x^*\circ \nu$ is bounded and finitely additive for every $x^* \in W$. This result
and the Rainwater-Simons theorem (see, e.g., \cite[Theorem~3.134]{fab-ultimo}) ensure that any James boundary of~$X$ has the OT property, see 
\cite[Proposition~2.9]{fer-alt-4} (cf. \cite[Remark~3.2(i)]{nyg-rod}).

The previous discussion makes clear that the particular
case of~$\ell_\infty$ is the most interesting one as the OT property is concerned.
In this paper we study the OT property in~$\ell_\infty$ and some aspects of the Banach lattices of real-valued functions which are integrable with respect to 
countably additive $\ell_\infty$-valued measures. The paper is organized as follows.

In Section~\ref{section:preliminaries} we introduce some terminology and preliminary facts.

In Section~\ref{section:OT} we focus on the OT property. We begin with some basic results in arbitrary Banach spaces, including an application to
the factorization of vector measures and their integration operators (Proposition~\ref{proposition:Susumu}). The core of this section is devoted to studying the OT property
in~$\ell_\infty$. Our main result here (Theorem~\ref{theorem:ellinfty-uncountable-2})
states that any subset of~$\ell_\infty^*$ having the OT property contains a copy of the usual basis of $\ell_1(\mathfrak{c})$, where $\mathfrak{c}$ stands for 
the cardinality of the continuum (i.e., $\mathfrak{c}=|\mathbb{R}|$). Some consequences of this result are given. 
Let us mention that the existence of copies of the usual basis of~$\ell_1(\mathfrak{c})$ inside James boundaries (which always have the OT property) 
has been studied thoroughly (see \cite{god5,gra-her} and the references therein).

In Section~\ref{section:L1} we study some structural properties of the Banach lattice $L_1(\nu)$ of real-valued
functions which are integrable with respect to a given $\nu\in {\rm ca}(\Sigma,X)$, where $\Sigma$ is a $\sigma$-algebra and $X$ is a Banach space.
These spaces play an important role in Banach lattice and operator theory. It is known that any order continuous Banach lattice having a weak order unit 
is lattice-isometric to a space of the form~$L_1(\nu)$. 
Suppose that $\nu$  has separable range or equivalently that $\nu$ takes its values in~$\ell_\infty$. Does it follow that $L_1(\nu)$ is separable?   
The answer is in the negative, in general, even if $\nu$ is a finite measure (see Subsection 2.4).  
However, the answer is in the affirmative with an additional assumption that $L_1(\nu)^*$ is order continuous as asserted in Theorem~\ref{theorem:countable-norming-general}.  
 We shall instead prove a more general fact: Theorem~\ref{theorem:general}, which asserts  that  if $E$ is a Banach lattice admitting a countable positively norming set
and both $E$ and~$E^*$ are order continuous, then $E$ is separable. In particular, if $\mu$ is a finite measure 
for which $L_1(\mu)$ is non-separable and $1<p<\infty$, then $L_p(\mu)$ is not isomorphic to $L_1(\nu)$ for any $\nu$ with separable range
(Example~\ref{example:Lp}).

\section{Terminology and preliminaries}\label{section:preliminaries}

Given a set $S$, we denote by $\mathcal{P}(S)$ its power set, that is, the set of all subsets of~$S$.
The cardinality of $S$ is denoted by~$|S|$. The {\em density character} of a topological space~$(T,\mathfrak{T})$, denoted by~${\rm dens}(T,\mathfrak{T})$
or simply ${\rm dens}(T)$, is the minimal cardinality of a $\mathfrak{T}$-dense subset of~$T$.

\subsection{Banach spaces}\label{subsection:Banach}

All Banach spaces considered in this paper are real. An {\em operator} is a continuous linear map between Banach spaces. Let $X$ be a Banach space.
The norm of~$X$ is denoted by $\|\cdot\|_X$ or simply~$\|\cdot\|$. 
The closed unit ball of~$X$ is $B_X:=\{x\in X:\|x\|\leq 1\}$. 
The weak (resp., weak$^*$) topology on~$X$ (resp., $X^*$) is denoted by~$w$ (resp., $w^*$).
By a {\em subspace} of~$X$ we mean a norm-closed linear subspace. In almost all cases we will deal with norm-closed linear subspaces, so we prefer to use
such an abridged terminology; when norm-closedness is not assumed, we 
use the term {\em linear subspace} unless otherwise stated.
 By a {\em projection} from~$X$ onto a subspace~$Y \sub X$
we mean an operator $P:X\to X$ such that $P(X)=Y$ and $P$ is the identity when restricted to~$Y$.
The convex hull and linear span of a set~$D\sub X$ are denoted by ${\rm co}(D)$ and ${\rm span}(D)$, respectively, and 
their closures are denoted by $\overline{{\rm co}}(D)$ and $\overline{{\rm span}}(D)$.
A set $B \sub B_{X^*}$ is said to be {\em norming} if there is a constant $c>0$ such that 
$\|x\|\leq c \sup_{x^*\in B}|x^*(x)| $ for every $x\in X$.
A set $B \sub B_{X^*}$ is said to be a {\em James boundary} of~$X$ if for every $x\in X$ there is $x^*\in B$ such that $\|x\|=x^*(x)$.
If $X$ is a Banach lattice, then its positive cone is $X^+=\{x\in X:x\geq 0\}$.
Given a compact Hausdorff topological space~$K$, we denote by $C(K)$ the Banach space of all real-valued continuous functions on~$K$, equipped
with the supremum norm.

\subsection{Banach function spaces}\label{subsection:Bfs}

Let $(\Omega,\Sigma,\mu)$ be a finite measure space. 
To define Banach function spaces, we consider linear subspaces, not necessarily norm-closed, of $L_1(\mu)$.
A Banach space~$E$ is said to be
a \emph{Banach function space} (or a \emph{K\"{o}the function space}) over $(\Omega,\Sigma,\mu)$ if the following conditions hold: 
\begin{enumerate}
\item[(i)] $E$ is a linear subspace of $L_1(\mu)$;
\item[(ii)] if $f\in L_0(\mu)$ and $|f| \leq |g|$ $\mu$-a.e. for some $g \in E$, then $f \in E$ and $\|f\|_{E} \leq \|g\|_{E}$;
\item[(iii)] the characteristic function $\chi_A$ of each $A \in \Sigma$ belongs to~$E$. 
\end{enumerate}
In this case, $E$ is a Banach lattice when endowed with the $\mu$-a.e. order and the inclusion map $E \to L_1(\mu)$ is an operator.
The {\em K\"{o}the dual} of~$E$ is 
$$
	E':=\{g\in L_1(\mu): \, fg \in L_1(\mu) \text{ for all $f\in E$}\}.
$$
Any $g \in E'$ gives raise to a functional $\varphi_g\in E^*$ defined by 
$\varphi_g(f):=\int_\Omega fg \, d\mu$ for all $f\in E$. It is known that
$E$ is order continuous if and only if $E^*=\{\varphi_g:g\in E'\}$ (see, e.g., \cite[p.~29]{lin-tza-2}).

\subsection{$L_1$ of a vector measure}\label{subsection:L1spaces}

Let $(\Omega,\Sigma)$ be a measurable space, let $X$ be a Banach space and let $\nu \in {\rm ca}(\Sigma,X)$. 
A set $A\in \Sigma$ is said to be {\em $\nu$-null} if 
$\nu(B)=0$ for every $B\in \Sigma$ with~$B\sub A$.
The family of all $\nu$-null sets is denoted by $\mathcal{N}(\nu)$. By a {\em Rybakov control measure} of~$\nu$
we mean a finite measure of the form $\mu=|x^*\circ \nu|$ (the variation of $x^*\circ \nu$) for some $x^*\in X^*$ such that 
$\mathcal{N}(\mu) = \mathcal{N}(\nu)$ (see, e.g., \cite[p.~268, Theorem~2]{die-uhl-J} for a proof of the existence
of Rybakov control measures). 
A $\Sigma$-measurable function $f:\Omega \to \mathbb{R}$ is called {\em $\nu$-integrable} if $f\in L_1(|x^*\circ \nu|)$ for all $x^*\in X^*$ and,
for each $A\in \Sigma$, there is $\int_A f \, d\nu\in X$
such that
$$
	x^*\left ( \int_A f \, d\nu \right)=\int_A f\,d(x^*\circ \nu)  	\quad\text{for all $x^*\in X^*$}.
$$
By identifying functions which coincide $\nu$-a.e., the set $L_1(\nu)$ of all (equivalence classes of) $\nu$-integrable functions  
is a Banach space with the norm
$$
	\|f\|_{L_1(\nu)}:=\sup_{x^*\in B_{X^*}}\int_\Omega |f|\,d|x^*\circ \nu|.
$$
The \emph{integration operator} is the (norm one) operator $I_\nu: L_1(\nu)\to X$ defined by  
$$
	I_\nu(f):=\int_\Omega f\, d\nu
	\quad\text{for all $f\in L_1(\nu)$.}
$$ 
If $\mu$ is any Rybakov control measure of~$\nu$, then $L_1(\nu)$ is
a Banach function space over $(\Omega,\Sigma,\mu)$. As a Banach lattice, $L_1(\nu)$ is order continuous and
has a weak order unit (the function $\chi_\Omega$). Conversely, any order continuous Banach lattice
having a weak order unit is lattice-isometric to the $L_1$ space of a countably additive vector measure defined
on a $\sigma$-algebra. Indeed, on the one hand, such a
Banach lattice is lattice-isometric to a Banach function space~$E$ over some finite measure space $(\Omega,\Sigma,\mu)$
(see, e.g., \cite[Theorem~1.b.14]{lin-tza-2}). On the other hand, thanks to the order continuity of~$E$, the map $\nu:\Sigma \to E$ given by $\nu(A):=\chi_A$
for all $A\in \Sigma$ is countably additive and one has $E=L_1(\nu)$ (see \cite[Theorem~8]{cur1}, \cite[Proposition~2.4(vi)]{dep-alt}).
We refer the reader to \cite[Chapter~3]{oka-alt}
for more information on the $L_1$ space of a vector measure.

\subsection{The usual measure on~$\{0,1\}^I$}\label{subsection:Cantor}

Let $I$ be a non-empty set. For each $i\in I$ we denote by $\pi_i:\{0,1\}^I \to \{0,1\}$
the $i$th-coordinate projection. Let $\Sigma_I$ be the $\sigma$-algebra on~$\{0,1\}^I$ generated by all the sets of the form $\bigcap_{i\in F}\pi_i^{-1}(w(i))$,
where $F \sub I$ is finite and $w\in\{0,1\}^F$. The usual product probability measure on~$\{0,1\}^I$, denoted by~$\lambda_I$, is 
defined on~$\Sigma_I$ and satisfies $\lambda_I(\pi_i^{-1}(\{0\}))=\lambda_I(\pi_i^{-1}(\{1\}))=\frac{1}{2}$
for all $i\in I$. For simplicity, we just call $\lambda_I$ the {\em usual measure on~$\{0,1\}^I$}. 
We have ${\rm dens}(L_1(\lambda_I))=|I|$ if $I$ is infinite.
In particular, $L_1(\lambda_I)$ is not separable whenever $I$ is uncountable. 
We refer the reader to \cite[\S254]{freMT-2} for more information on infinite product measures
and the usual measure on~$\{0,1\}^I$.

\subsection{Measure algebras}\label{subsection:MeasureAlgebras}

Let $(\Omega,\Sigma,\mu)$ be a probability space.
We consider the equivalence relation on~$\Sigma$ defined by $A \sim B$ if and only if $\mu(A\triangle B)=0$.
The set of equivalence classes, denoted by $\Sigma/\mathcal{N(\mu)}$, becomes a measure algebra
when equipped with the usual Boolean algebra operations
and the functional defined by $\mu^{\bullet}(A^{\bullet}):=\mu(A)$ for all $A\in \Sigma$, where $A^{\bullet}\in \Sigma/\mathcal{N}(\mu)$
denotes the equivalence class of~$A$. Given another probability space $(\Omega_0,\Sigma_0,\mu_0)$,
the measure algebras of~$\mu$ and~$\mu_0$ are said to be {\em isomorphic}
if there is a Boolean algebra isomorphism
$$
	\theta: \Sigma/\mathcal{N}(\mu) \to \Sigma_0/\mathcal{N}(\mu_0)
$$
such that $\mu_0^{\bullet}\circ \theta=\mu^{\bullet}$.
In this case, there is a lattice isometry $\Phi:L_1(\mu)\to L_1(\mu_0)$
such that $\int_\Omega f\, d\mu=\int_{\Omega_0}\Phi(f)\, d\mu_0$ and $\Phi(f\chi_B)=\Phi(f)\chi_{C}$ 
whenever $f\in L_1(\mu)$ and $\theta(B^{\bullet})=C^{\bullet}$.
For more information on measure algebras, see \cite{fre14}.

\section{The Orlicz-Thomas property}\label{section:OT}

\subsection{The OT property in arbitrary Banach spaces}
Throughout this subsection $X$ is a Banach space. 
We begin with an observation:

\begin{proposition}\label{proposition:OTimpliestotal}
If $W \sub X^*$ has the OT property, then $W$ is total.
\end{proposition}
\begin{proof}
If $W$ is not total, then there is $x\in X\setminus \{0\}$ such that $x^*(x)=0$ for all $x^*\in W$. Let $\xi:\mathcal{P}(\N)\to [0,1]$
be a finitely additive measure which is not countably additive. 
Then there is a disjoint sequence $(A_n)_{n\in \N}$ in $\mathcal{P}(\N)$
such that the sequence $(\xi(\bigcup_{m>n}A_m))_{n\in \N}$ does not converge to~$0$.
Define $\nu:\mathcal{P}(\N) \to X$ by $\nu(A):=\xi(A)x$ for all $A \sub \N$. For each $x^*\in W$ we have
$(x^*\circ \nu)(A)=0$ for all $A\sub \N$, hence $x^*\circ \nu\in {\rm ca}(\mathcal{P}(\N))$. Since $\|\nu(\bigcup_{m>n}A_m)\|=\xi(\bigcup_{m>n}A_m)\|x\|$ for every $n\in \N$, 
we have $\nu\not\in {\rm ca}(\mathcal{P}(\N),X)$.
\end{proof}

The following proposition is straightforward.

\begin{proposition}\label{proposition:convex}
Let  $W \sub X^*$. The following statements are equivalent: 
\begin{enumerate}
\item[(i)] $W$ has the OT property. 
\item[(ii)] ${\rm co}(W)$ has the OT property.
\item[(iii)] ${\rm span}(W)$ has the OT property.
\end{enumerate} 
\end{proposition}

As usual, $\omega_1$ denotes the first uncountable ordinal.
Given any  set $D \sub X^*$, we denote by $S_1(D) \sub X^*$ the set of all limits of $w^*$-convergent 
sequences contained in~$D$. For any ordinal $\alpha\leq \omega_1$, we define $S_\alpha(D)$
by transfinite induction as follows: 
\begin{itemize}
\item $S_0(D):=D$, 
\item $S_{\alpha}(D):=S_1(S_\beta(D))$ if $\alpha=\beta+1$ for some ordinal $\beta<\omega_1$, 
\item $S_{\alpha}(D):=\bigcup_{\beta<\alpha}S_\beta(D)$
if $\alpha$ is a limit ordinal. 
\end{itemize}
Then $S_{\omega_1}(D)$ is the smallest $w^*$-sequentially closed
subset of~$X^*$ containing~$D$. In general, we have 
$$
	D \sub \overline{D}^{\|\cdot\|} \sub S_1(D) \sub S_{\omega_1}(D) \sub \overline{D}^{w^*}.
$$

\begin{proposition}\label{proposition:OT-sequential}
Let $W \sub X^*$. The following statements are equivalent: 
\begin{enumerate}
\item[(i)] $W$ has the OT property. 
\item[(ii)] $\overline{W}^{\|\cdot\|}$ has the OT property.
\item[(iii)] $S_1(W)$ has the OT property.
\item[(iv)] $S_{\omega_1}(W)$ has the OT property.
\item[(v)] $\overline{W}^{w}$ has the OT property.
\end{enumerate}
\end{proposition}
\begin{proof}
The implications (i)$\impli$(ii)$\impli$(iii)$\impli$(iv) and (i)$\impli$(v) are obvious.

(iv)$\impli$(i): Suppose that $S_{\omega_1}(W)$ has the OT property.
Let $\nu:\Sigma \to X$ be a map defined on a $\sigma$-algebra~$\Sigma$ such that $x^* \circ\nu \in {\rm ca}(\Sigma)$ for all $x^*\in W$.
Bearing in mind the Vitali-Hahn-Saks-Nikod\'{y}m theorem (see, e.g., \cite[p.~24, Corollary~10]{die-uhl-J}), 
a standard transfinite induction argument shows that, for each ordinal $\alpha\leq \omega_1$, we have $x^*\circ \nu\in {\rm ca}(\Sigma)$  for all $x^*\in S_\alpha(W)$.
In particular, this holds for $\alpha=\omega_1$ and so $\nu\in {\rm ca}(\Sigma,X)$.

(v)$\impli$(i): Since $\overline{W}^w \sub \overline{{\rm co}(W)}^{w}=\overline{{\rm co}(W)}^{\|\cdot\|}$, the set $\overline{{\rm co}(W)}^{\|\cdot\|}$ has the OT property.
By the equivalence (i)$\Leftrightarrow$(ii) applied to ${\rm co}(W)$, this set has the OT property and so does~$W$ (by Proposition~\ref{proposition:convex}). 
\end{proof}

The following proposition gives an operator theoretic reformulation of the OT property.
Given a $\sigma$-algebra~$\Sigma$, the set ${\rm ca}(\Sigma)$ is a Banach space
when equipped with the variation norm. It is known that a set $H \sub {\rm ca}(\Sigma)$ is relatively weakly compact if and only if it is bounded
and there is a non-negative $\mu_0 \in {\rm ca}(\Sigma)$ such that $H$ is {\em uniformly $\mu_0$-continuous}, i.e.,
for every $\epsilon>0$ there is $\delta>0$ such that $\sup_{\mu\in H}|\mu(A)|\leq \epsilon$ for every $A\in \Sigma$ satisfying $\mu_0(A)\leq \delta$
(see, e.g., \cite[p.~92, Theorem~13]{die-J}).

\begin{proposition}\label{proposition:OT-operator}
Let $W \sub X^*$ be a subspace and let $\nu:\Sigma \to X$ be a map defined on a $\sigma$-algebra~$\Sigma$ such that $x^*\circ \nu \in {\rm ca}(\Sigma)$ 
for all $x^*\in W$. Then the map
$$
	T:W \to {\rm ca}(\Sigma),
	\quad
	T(x^*):=x^*\circ \nu,
$$
is an operator and the following statements hold:
\begin{enumerate}
\item[(ii)] If $\nu\in {\rm ca}(\Sigma,X)$, then $T$ is weakly compact.
\item[(iii)] If $T$ is weakly compact and $B_{W} \sub B_{X^*}$ is norming, then $\nu \in {\rm ca}(\Sigma,X)$.
\end{enumerate}
\end{proposition}
\begin{proof} A routine application of the Closed Graph Theorem ensures that $T$ is an operator.

(i) The fact that $\nu$ is countably additive ensures the existence of a non-negative $\mu_0 \in {\rm ca}(\Sigma)$ such that
$\{x^*\circ \nu: x^*\in B_{X^*}\} \supseteq T(B_W)$ is uniformly $\mu_0$-continuous
(see, e.g., \cite[p.~14, Corollary~6]{die-uhl-J}).

(ii) Since $B_W$ is absolutely convex and norming, we have 
\begin{equation}\label{eqn:norm}
	\overline{B_W}^{w^*} \supseteq c B_{X^*}
\end{equation}
for some $c>0$, by the Hahn-Banach separation theorem. Fix a non-negative  
$\mu_0\in {\rm ca}(\Sigma)$ such that $T(B_W)$ is uniformly $\mu_0$-continuous.
Observe that $\nu$ is finitely additive because $W$ is total (bear in mind that $B_W$ is norming and so total) and $x^*\circ \nu$ is finitely additive for every $x^*\in W$.  
To prove that $\nu$ is countably additive it suffices to check that it is $\mu_0$-continuous. 
Fix $\epsilon>0$. Choose $\delta>0$ such that
$$
	|x^*(\nu(A))|\leq c\epsilon
	\quad
	\text{for every $A\in \Sigma$ with $\mu_0(A)\leq \delta$ and for every $x^*\in B_W$}.
$$
Clearly, the previous inequality is also valid for all $x^*\in \overline{B_W}^{w^*}$ and then \eqref{eqn:norm}
implies that
$$
	\|\nu(A)\|=\sup_{x^*\in B_{X^*}}|x^*(\nu(A)) |\leq \epsilon
	\quad
	\text{for every $A\in \Sigma$ with $\mu_0(A)\leq \delta$}.
$$
Therefore, $\nu$ is $\mu_0$-continuous and so it is countably additive.
\end{proof}

\begin{proposition}\label{proposition:injective}
Let $T: X \to Y$ be an operator between Banach spaces and let $\nu:\Sigma \to X$ be a map defined on a $\sigma$-algebra~$\Sigma$. The following statements hold:
\begin{enumerate}
\item[(i)] $T^*(B_{Y^*})$ is a $w^*$-compact subset of~$X^*$.
\item[(ii)] $T^*(B_{Y^*})$ is total if and only if $T$ is injective.
\item[(iii)]  $T\circ \nu \in {\rm ca}(\Sigma,Y)$ if and only if $x^* \circ \nu \in {\rm ca}(\Sigma)$ for all $x^*\in T^*(B_{Y^*})$.
\item[(iv)] Suppose that $T^*(B_{Y^*})$ has the OT property. Then $\nu \in {\rm ca}(\Sigma,X)$ if and only if $T\circ \nu \in {\rm ca}(\Sigma,Y)$.
\item[(v)] If $T^{**}$ is injective, then $T^*(B_{Y^*})$ has the OT property.
\end{enumerate}
\end{proposition}
\begin{proof} (i) follows from the $w^*$-compactness of~$B_{Y^*}$ and the $w^*$-to-$w^*$ continuity of~$T^*$. (ii) is immediate.
(iii) follows at once from the Orlicz-Pettis theorem. (iv) is a consequence of (iii).

To prove (v), note  that the injectivity of~$T^{**}$ is equivalent (via the Hahn-Banach separation theorem)
to the norm denseness  of $T^*(Y^*)$ in~$X^*$. From the Orlicz-Pettis theorem and Proposition~\ref{proposition:OT-sequential} it follows that
$T^*(Y^*)$ has the OT property. Since ${\rm span}(T^*(B_{Y^*}))= T^*(Y^*)$, we can apply Proposition~\ref{proposition:convex}
to conclude that $T^*(B_{Y^*})$ has the OT property.
\end{proof}

Typical examples of non-isomorphic embeddings having injective biadjoints are the operators associated to the 
Davis-Figiel-Johnson-Pe{\l}czy\'nski factorization (see, e.g., \cite[Theorem~5.37]{ali-bur}).
The following simple example shows that the conclusion of Proposition~\ref{proposition:injective}(iv) can fail 
for arbitrary injective operators. 

\begin{example}\label{example:injective}
\rm Let $\nu: \mathcal{P}(\N) \to \ell_\infty$ be the finitely additive measure
defined by $\nu(A):=\chi_A$ for all $A\sub \N$ and let $T:\ell_\infty \to \ell_1$ be the injective operator 
defined by $T((x_n)_{n\in \N}):=(2^{-n}x_n)_{n\in \N}$ for all $(x_n)_{n\in \N}\in \ell_\infty$. Observe that 
$T\circ\nu$ is finitely additive and
$$
	\|(T \circ \nu)(A)\|=\sum_{n\in A}2^{-n}
	\quad\text{for all $A \sub \N$}, 
$$
hence $T \circ \nu \in {\rm ca}(\mathcal{P}(\N),\ell_1)$. However, $\nu\not\in {\rm ca}(\mathcal{P}(\N),\ell_\infty)$.
\end{example}

Let $T: X \to Y$ be an injective operator between Banach spaces, let $(\Omega,\Sigma)$ be a measurable space
and suppose that the integration operator $I_\nu$ of $\nu \in {\rm ca}(\Sigma,Y)$ factors as
$$
	\xymatrix@R=2pc@C=2pc{
	L_1(\nu) \ar[rr]^{I_\nu} \ar[dr]_S &      & Y \\
	 & X \ar[ur]_{T} &
	}
$$
for some operator~$S$. Define $\tilde{\nu}(A):=S(\chi_A)$ for all $A\in \Sigma$.
In \cite[Theorem~3.7]{nyg-rod} it was proved that $\tilde{\nu}\in {\rm ca}(\Sigma,X)$ satisfies:
\begin{enumerate}
\item[(a)] $\nu=T\circ \tilde{\nu}$ and $\mathcal{N}(\nu)=\mathcal{N}(\tilde{\nu})$;
\item[(b)] $L_1(\nu)=L_1(\tilde{\nu})$ (with equivalent norms); and
\item[(c)] $S=I_{\tilde{\nu}}$.
\end{enumerate}
This result improves \cite[Lemma~3.1]{rod15} in which (b) was obtained via the Diestel-Faires Theorem~\ref{theorem:DF}
under the additional assumption that $X$ does not contain subspaces isomorphic to~$\ell_\infty$.
In a similar spirit, we have:

\begin{proposition}\label{proposition:Susumu}
Let $T: X \to Y$ be an operator between Banach spaces such that $T^*(B_{Y^*})$ has the OT property.
Let $(\Omega,\Sigma)$ be a measurable space and let $\nu \in {\rm ca}(\Sigma,Y)$ be such that
$I_\nu(L_1(\nu)) \sub T(X)$. Then there is $\tilde{\nu}\in {\rm ca}(\Sigma,X)$ such that:
\begin{enumerate}
\item[(a)] $\nu=T\circ \tilde{\nu}$  and $\mathcal{N}(\nu)=\mathcal{N}(\tilde{\nu})$;
\item[(b)] $L_1(\nu)=L_1(\tilde{\nu})$ (with equivalent norms); and
\item[(c)] $I_\nu=T \circ I_{\tilde{\nu}}$.
\end{enumerate}
\end{proposition}
\begin{proof}
Observe that $T$ is injective (by Propositions~\ref{proposition:OTimpliestotal} and~\ref{proposition:injective}(ii)).
For each $A\in \Sigma$ we have $\nu(A)=I_\nu(\chi_A)\in T(X)$, so there is a unique $\tilde{\nu}(A)\in X$ such that
$$
	T(\tilde{\nu}(A))=\nu(A).
$$
The so-defined map $\tilde\nu:\Sigma \to X$ satisfies $\nu=T\circ \tilde{\nu}$ and belongs to ${\rm ca}(\Sigma,X)$ because
$T^*(B_{Y^*})$ has the OT property and
$T^*(y^*)\circ \tilde{\nu}=y^* \circ T \circ \tilde{\nu}=y^*\circ \nu \in {\rm ca}(\Sigma)$ for all $y^*\in B_{Y^*}$.

By \cite[Lemma~3.27]{oka-alt} we have $\mathcal{N}(\nu)=\mathcal{N}(\tilde{\nu})$, the inclusion $L_1(\tilde{\nu})\subseteq L_1(\nu)$
and the equality $I_\nu=T \circ I_{\tilde{\nu}}$ on $L_1(\tilde{\nu})$.

To prove the reverse inclusion $L_1(\nu)\subseteq L_1(\tilde{\nu})$, let
$f \in  L_1(\nu)$. The fact that $I_\nu(L_1(\nu)) \sub T(X)$ enables us to define a finitely additive set function $\eta: \Sigma
\to X$ such that $T( \eta (A) )=I_\nu(f\chi_A)= \int_Af\,d\nu$ for every $A \in \Sigma$. Then, Proposition \ref{proposition:injective}(iv) ensures that  
$\eta$ is countably additive because $T^*(B_{Y^*})$ has the OT property and the indefinite integral $A\mapsto \int_Af\,d\nu$ on $\Sigma$ is countably additive.  
Given $n \in \N$, let $A_n: = |f|^{-1} ([0,n])\in \Sigma$ and $f_n:= f\chi_{A_n}$.  Fix $A \in\Sigma$. 
Each $f_n$ is bounded and $\Sigma$-measurable, hence it is $\tilde{\nu}$-integrable and, moreover, it satisfies
\begin{equation}\label{eqn:Susumu}
	\int_A f_n\,d\tilde{\nu} = \eta(A\cap A_n).
\end{equation}
Indeed, let $(s_k^{(n)})_{k\in \N}$ be a sequence of $\Sigma$-simple functions which are uniformly convergent to~$f_n$ as $k\to \infty$. 
Since $\nu=T\circ \tilde{\nu}$, it follows that
\begin{eqnarray*}
 	T\left(\int_A f_n\,d\tilde{\nu} \right) & = &  T\left(\lim_{k\to\infty}\int_A s_k^{(n)}\,d\tilde{\nu} \right)  =  
 	\lim_{k\to\infty}T\left(\int_A s_k^{(n)}\,d\tilde{\nu} \right) \\ & =&  \lim_{k\to\infty}\int_A s_k^{(n)}\,d\nu
  	=\int_A f_n\,d\nu= \int_{A\cap A_n} f \,d\nu = T\left(\eta (A\cap A_n) \right).
\end{eqnarray*}
This verifies \eqref{eqn:Susumu} as $T$ is injective.  Now \eqref{eqn:Susumu} together with countable additivity of~$\eta$ imply that $ \lim_{n\to \infty} \int_A f_n\,d\tilde{\nu} =  \lim_{n\to \infty}\eta(A \cap A_n) =\eta(A)$ (as $(A_n)_{n\in \N}$ is increasing with union~$\Omega$).  
Since this holds for an arbitrarily fixed $A \in \Sigma$ and since $\lim_{n\to\infty}f_n = f$ pointwise on $\Omega$, it follows from a result by Lewis
(see, e.g., \cite[Theorem 3.5]{oka-alt}) that $ f \in L^1(\tilde{\nu})$.  Therefore we have proved $L_1(\nu)\subseteq L_1(\tilde{\nu})$ and hence, (c) holds.
The Closed Graph Theorem can be used to show
that both inclusions $L_1(\nu)\subseteq L_1(\tilde{\nu})$ and $L_1(\tilde{\nu})\subseteq L_1(\nu)$ are continuous, so that the norms 
of $L_1(\nu)$ and~$L_1(\tilde{\nu})$ are equivalent.
\end{proof}

We finish this subsection with two results showing that the study of countable additivity of vector measures in arbitrary Banach spaces 
can be reduced somehow to the $\ell_\infty$-valued case.

\begin{proposition}\label{proposition:tilde}
Let $W \sub B_{X^*}$ and let $i_W: X \to \ell_\infty(W)$ be the operator defined by
$$
	i_W(x):=\big(x^*(x)\big)_{x^* \in W}
	\quad \text{for all $x\in X$}.
$$ 
Let $\nu:\Sigma \to X$ be a map defined on a $\sigma$-algebra~$\Sigma$ and define $\widehat \nu_W:=i_W\circ \nu$. 
Let us consider the following statements:
\begin{enumerate}
\item[(i)] $\nu \in {\rm ca}(\Sigma,X)$.
\item[(ii)] $\widehat{\nu}_W \in {\rm ca}(\Sigma,\ell_\infty(W))$.
\item[(iii)] $\varphi \circ \widehat \nu_W\in {\rm ca}(\Sigma)$ for every $\varphi \in \ell_1(W) \sub \ell_\infty(W)^*$.
\item[(iv)] $x^*\circ \nu\in {\rm ca}(\Sigma)$ for every $x^* \in W$.
\end{enumerate}
Then (i)$\impli$(ii)$\impli$(iii)$\Leftrightarrow$(iv). Moreover:
\begin{enumerate}
\item[(a)] If $W$ is $w^*$-compact, then (ii)$\Leftrightarrow$(iii)$\Leftrightarrow$(iv).
\item[(b)] If $i_W^{**}$ is injective, then (i)$\Leftrightarrow$(ii).
\end{enumerate}
\end{proposition}
\begin{proof}
The implications (i)$\impli$(ii)$\impli$(iii) are immediate. 

To prove (iii)$\Leftrightarrow$(iv), observe first that
for each $x^*\in W$ we have $x^*\circ \nu=e_{x^*}\circ \widehat \nu_W$, where $e_{x^*}\in \ell_1(W)$ 
is the vector defined by $e_{x^*}(z^*)=0$ for all $z^*\in W\setminus \{x^*\}$
and $e_{x^*}(x^*)=1$. Since $\ell_1(W)=\overline{{\rm span}(\{e_{x^*}:x^*\in W\})}^{\|\cdot\|}\sub \ell_\infty(W)^*$, the equivalence 
(iii)$\Leftrightarrow$(iv) is a consequence of the Vitali-Hahn-Saks-Nikod\'{y}m theorem (see, e.g., \cite[p.~24, Corollary~10]{die-uhl-J}).

(a) If $W$ is $w^*$-compact, then $i_W$ takes values in the subspace $C(W) \sub \ell_\infty(W)$ of all $w^*$-continuous real-valued functions on~$W$.
The set 
$$
	\Gamma:=\{\pm e_{x^*}|_{C(W)}: \, x^*\in W\} \sub B_{C(W)^*}
$$ 
is a James boundary of~$C(W)$ and so it has the OT property, as we already mentioned in the Introduction. 
Finally, observe that (iv) is equivalent to saying that $\gamma \circ \widehat \nu_W \in {\rm ca}(\Sigma)$ for every $\gamma \in \Gamma$.

(b) This follows at once from Proposition~\ref{proposition:injective}(iv-v).
\end{proof}

Observe that if $W \sub B_{X^*}$ is norming, then the operator $i_W$ of Proposition~\ref{proposition:tilde}
is an isomorphic embedding.

\begin{proposition}\label{proposition:diagram}
Let $W \sub B_{X^*}$ be a norming set and let 
$\nu:\Sigma \to X$ be a map defined on a $\sigma$-algebra~$\Sigma$ such that
$x^*\circ \nu\in {\rm ca}(\Sigma)$ for every $x^* \in W$. If $\nu \not\in {\rm ca}(\Sigma,X)$,
then there is a countable set $W_0 \sub W$ such that $\widehat{\nu}_{W_0}\not\in {\rm ca}(\Sigma,\ell_\infty(W_0))$, and we have a commutative diagram
$$
	\xymatrix@R=2.5pc@C=2.5pc{
	&      & \ell_\infty(W) \ar[dd]^{P_{W_0}} \\
	\Sigma \ar[r]^{\nu} \ar[urr]^{\widehat\nu_W} \ar[drr]_{\widehat\nu_{W_0}} & X \ar[dr]^{i_{W_0}} \ar[ur]_{i_{W}} &  \\
	 &     & \ell_\infty(W_0) \\
	}
$$
where $P_{W_0}$ is the operator defined by $P_{W_0}(u):=u|_{W_0}$ for all $u\in \ell_\infty(W)$. 
\end{proposition}
\begin{proof} Observe that $\nu$ is finitely additive. Since $\nu$ is not countably additive, we can take 
a sequence $(A_i)_{i\in \N}$ of pairwise disjoint elements of~$\Sigma$ such that the sequence 
$(\nu(\bigcup_{i>n}A_i))_{n\in \N}$ does not converge to~$0$. Write
$B_n:=\bigcup_{i>n}A_i$ for all $n\in \N$. Since $W$ is norming, there is a constant $k>0$ such that
for each $n\in \N$ there is $x_n^* \in W$ such that 
$|x_n^*(\nu(B_n))| \geq k \|\nu(B_n)\|_X$. Let $W_0:=\{x_n^*:n\in \N\}$. Then
$$
	\|\widehat \nu_{W_0} (B_n)\|_{\ell_\infty(W_0)} \ge |x_n^*(\nu(B_n))| \geq  k \|\nu(B_n)\|_X
	\quad
	\text{for all $n\in \N$},
$$
hence the sequence $(\widehat\nu_{W_0}(B_n))_{n\in \N}$ does not converge to~$0$ in~$\ell_\infty(W_0)$. It follows
that $\widehat{\nu}_{W_0}\not\in {\rm ca}(\Sigma,\ell_\infty(W_0))$. The last statement is immediate.
\end{proof}

\subsection{The OT property in~$\ell_\infty$}\label{subsection:ell_infty}

As noted in the Introduction, the finitely additive map $\nu:\mathcal{P}(\N)\to \ell_\infty$
defined by $\nu(A):=\chi_A$ for $A \sub \N$ is not countably additive while $\pi_n \circ\nu\in {\rm ca}(\mathcal{P}(\N))$
for each coordinate functional $\pi_n\in \ell_\infty^*$. This example has been used to see that the set $\{\pi_n:n\in \N\}\sub \ell_\infty^*$
fails to have the OT property. To provide further examples of the same nature, we shall first determine the form of general $\ell_\infty$-valued
countably additive measures in Proposition~\ref{proposition:S} below. The proof uses a couple of results which shall also be needed later. The first
one goes back to Bartle, Dunford and Schwartz~\cite{bar-alt} (cf. \cite[p.~14, Corollary~7]{die-uhl-J}) while
the second one is due to Rosenthal~\cite{ros-J-4} (cf. \cite[p.~252, Theorem~13]{die-uhl-J}).

\begin{theorem}[Bartle-Dunford-Schwartz]\label{theorem:BDS}
Let $\Sigma$ be a $\sigma$-algebra, let $X$ be a Banach space and let $\nu \in {\rm ca}(\Sigma,X)$. Then the range of~$\nu$, that is, the set
$$
	\mathcal{R}(\nu):=\{\nu(A): \, A\in \Sigma\}
$$
is relatively weakly compact in~$X$.
\end{theorem}

\begin{theorem}[Rosenthal]\label{theorem:Rosenthal-separable}
Any weakly compact subset of $\ell_\infty$ is norm-separable.
\end{theorem}

\begin{proposition}\label{proposition:S}
Let $\nu:\Sigma \to \ell_\infty$ be a map defined on a $\sigma$-algebra~$\Sigma$ such that $\{\pi_n\circ \nu:n\in \N\} \sub {\rm ca}(\Sigma)$.
The following conditions are equivalent:
\begin{enumerate}
\item[(i)] $\nu \in {\rm ca}(\Sigma,\ell_\infty)$.
\item[(ii)] $\{\nu(A): A\in \Sigma\}$ is relatively weakly compact in~$\ell_\infty$.
\item[(iii)] $\{\nu(A): A\in \Sigma\}$ is norm-separable in~$\ell_\infty$.
\item[(iv)] $\{\pi_n\circ \nu:n\in \N\}$ is a uniformly countably additive subset of~${\rm ca}(\Sigma)$.
\item[(v)] $\{\pi_n\circ \nu:n\in \N\}$ is relatively weakly compact in~${\rm ca}(\Sigma)$.
\item[(vi)] There exists a non-negative $\mu\in {\rm ca}(\Sigma)$ such that $\{\pi_n\circ \nu:n\in \N\}$
is uniformly $\mu$-continuous.
\end{enumerate}
\end{proposition}
\begin{proof}
The implications (i)$\impli$(ii) and (ii)$\impli$(iii) follow from Theorems~\ref{theorem:BDS} and~\ref{theorem:Rosenthal-separable}, respectively.

To prove (iii)$\impli$(i), note that (iii) implies that $\nu$ takes values in a separable subspace $X$ of~$\ell_\infty$. Since $X$ contains no subspace isomorphic to~$\ell_\infty$
and the set of restrictions $\{\pi_n|_X:n\in \N\}\sub X^*$ is total (for~$X$), the Diestel-Faires Theorem~\ref{theorem:DF}(ii) implies 
that $\nu$ is countably additive.

The equivalence (i)$\Leftrightarrow$(iv) is clear.

For each $A\in \Sigma$ we have $\sup_{n\in \N}|(\pi_n\circ \nu)(A)|=\|\nu(A)\|_{\ell_\infty}<\infty$. Therefore, 
the Nikod\'{y}m boundedness theorem (see, e.g., \cite[p.~14, Theorem~1]{die-uhl-J}) applies to conclude that
$\{\pi_n\circ \nu:n\in \N\}$ is bounded in ${\rm ca}(\Sigma)$. So, the equivalences (iv)$\Leftrightarrow$(v)$\Leftrightarrow$(vi)
follow from a well known characterization of relatively weakly compact subsets of~${\rm ca}(\Sigma)$ (see, e.g., \cite[p.~92, Theorem~13]{die-J}).
\end{proof}   

Motivated by condition~(vi) in Proposition~\ref{proposition:S} above, we present the following:

\begin{example}\label{example:S}
\rm Let $\lambda$ be the Lebesgue measure on the Lebesgue $\sigma$-algebra $\Sigma$ of~$[0,1]$. Take a norm-bounded
sequence $(f_n)_{n\in \N}$ in $L_1[0,1]$.
Let $\nu:\Sigma \to \ell_\infty$ be the map defined by
$$
	\nu(A):=\left(\int_A f_n \, d\lambda \right)_{n\in \N} \quad
	\text{for all $A\in \Sigma$}.
$$
According to the proof of Proposition~\ref{proposition:S} and the Dunford-Pettis theorem (see, e.g., \cite[p.~93]{die-J}), 
the map $\nu$ is countably additive if and only if $(f_n)_{n\in \N}$ is uniformly $\lambda$-integrable, that is,
$$
	\lim_{\lambda(A)\to 0}\sup_{n\in \N}\int_A|f_n|\, d\lambda=0.
$$
This holds when there is $g\in L_1[0,1]$ such that $|f_n|\leq g$ for all $n\in \N$. For further criteria for uniform integrability, see \cite{die9}, for example.

It is easy to check that the norm-bounded sequence $(n\chi_{[0,1/n)})_{n\in \N}$ in $L_1[0,1]$ is not uniformly $\lambda$-integrable.
The same holds for the sequence $(f_n)_{n\in \N}$ defined by $f_n(t):=\int_0^t n\chi_{[1-1/n,1)}(s)/(1-s) ds$ for $t\in[0,1]$ and $n\in \N$, \cite[p.~295]{oka-ric-3}.
More generally, \cite[Example~1]{bal-mur} provides a norm-bounded sequence in $L_1[0,1]$ whose restriction to any
$A\in \Sigma\setminus \mathcal{N}(\lambda)$ is not uniformly $\lambda$-integrable.
\end{example}

Exploiting the fact that the Banach spaces $\ell_\infty$ and $L_\infty[0,1]$ are isomorphic (see, e.g., \cite[Theorem~4.3.10]{alb-kal}), let us see
that $L_1[0,1]$ considered as a total subspace of $L_\infty[0,1]^*$ fails to have the OT property.

\begin{example}\label{example:Linfty}
\rm Let $\lambda$ and $\Sigma$ be as in Example~\ref{example:S}. Define $\nu: \Sigma \to L_\infty[0,1]$ by 
$\nu(A)= \chi_A$ for all $A \in \Sigma$. Then $\nu\not\in {\rm ca}(\Sigma,L_\infty[0,1])$ and
for each $\varphi\in L_1[0,1]$ we have $\varphi \circ \nu \in {\rm ca}(\Sigma)$, because $(\varphi\circ \nu)(A)=\int_A \varphi \, d\lambda$ for all $A\in \Sigma$.
So, $L_1[0,1]\subseteq L_\infty[0,1]^*$ fails to have the OT property.
\end{example}

It turns out that the existence of a norm-separable subset of~$X^*$ having the OT property prevents the Banach space $X$ from containing subspaces isomorphic 
to~$\ell_\infty$. This is asserted in Corollary~\ref{corollary:countableOT} below, whose proof requires our main result:

\begin{theorem}\label{theorem:ellinfty-uncountable-2}
Let $W \sub \ell_\infty^*$ be a set not containing sets equivalent to the canonical basis of~$\ell_1(\mathfrak{c})$. Then $W$ fails the OT property.
\end{theorem}

Given a Banach space~$X$ and a non-empty set~$I$, recall that a set $\{x_i:i\in I\}\sub X$
is said to be {\em equivalent to the canonical basis of~$\ell_1(I)$} if it is bounded and there is a constant $c>0$ such that
$$
	\left\|\sum_{i\in I} a_i x_{i}\right\|\geq c \sum_{i\in I}|a_i|
$$
for every $(a_i)_{i\in I}\in \ell_1(I)$. In this case, $\overline{{\rm span}}(\{x_i: i\in I\})$ is isomorphic to~$\ell_1(I)$.

The proof of Theorem~\ref{theorem:ellinfty-uncountable-2} requires some known facts and uses the following result of Talagrand
(see \cite[Th\'{e}or\`{e}me~4]{tal4}). Recall that the {\em cofinality} (denoted by~$cf(\kappa)$) of a cardinal~$\kappa$
is the smallest cardinal~$\kappa'$ such that $\kappa$ is the union of $\kappa'$ many sets of cardinality $<\kappa$.
Both $cf(\omega_1)$ and $cf(\mathfrak{c})$ are uncountable (see, e.g., \cite[Corollary~5.12]{jec}).

\begin{theorem}[Talagrand]\label{theorem:Talagrand}
Let $I$ be a set such that $|I|$ has 
uncountable cofinality. Let $X$ be a Banach space and let $D \sub X$ be a set such that $X=\overline{{\rm span}}(D)$.  
If $X$ contains a subspace isomorphic to~$\ell_1(I)$, then $D$ contains a set which is equivalent to
the canonical basis of~$\ell_1(I)$.
\end{theorem}

The following result is an application of~\cite[Lemma~1.1]{ros-J-4}:

\begin{lemma}\label{lemma:Rosenthal-l1}
Let $X$ be the $\ell_1$-sum of a family of Banach spaces $\{X_i:i\in I\}$ and, for each $i\in I$, let $\pi_i:X\to X_i$ be the canonical projection. Let $W \sub X$ be a subspace.
If the set $\{i\in I: \, \pi_i(W)\neq\{0\}\}$ contains a set $J$ such that~$|J|$ has uncountable cofinality, then $W$ contains a subspace isomorphic to~$\ell_1(J)$. 
\end{lemma}
\begin{proof}
For each $i\in J$ we fix $x_i\in B_W$ and $x_i^*\in B_{X_i^*}$ with $x_i^*(\pi_i(x_i)) \neq 0$. Define an operator $T:X\to \ell_1(J)$ by
$$
	T(x):=\big(x_i^*(\pi_i(x))\big)_{i\in J} \quad\text{for all $x\in X$}.
$$
For each $k\in \N$ we define $J_{k}:=\{i\in J: |x_i^*(\pi_i(x_i))|> \frac{1}{k}\}$, so that $J=\bigcup_{k\in \N}J_k$.
Since $|J|$ has uncountable cofinality, there is $k\in \N$ such that $|J_k|=|J|$. 

For each $i\in J$, let $\phi_i\in \ell_1(J)^*$ be the $i$th-coordinate functional. Since $J_k$ is contained
in the set
$$
	J':=\left\{i\in J: \, \big|\phi_i(T(x))\big|>\frac{1}{k} \, \text{ for some $x\in B_W$}\right\},
$$ 
we have $|J'|=|J|$. We can now apply \cite[Lemma~1.1]{ros-J-4} to conclude
that $W$ contains a subspace isomorphic to~$\ell_1(J)$. 
\end{proof}

Another ingredient for proving Theorem~\ref{theorem:ellinfty-uncountable-2} is Theorem~\ref{theorem:Kalton} below. It can be proved 
as \cite[Proposition~4]{kal74} (which can also be found in \cite[Theorem~2.5.4]{alb-kal}), bearing in mind that there is an almost disjoint family $\mathcal{A}$ 
of infinite subsets of~$\N$ with $|\mathcal{A}|=\mathfrak{c}$ (see, e.g., the proof of \cite[Lemma~2.5.3]{alb-kal}).

\begin{theorem}[Kalton]\label{theorem:Kalton}
Let $T:\ell_\infty \to \ell_\infty(I)$ be an operator, where $I$ is a non-empty set with $|I|<\mathfrak{c}$. If $T$ vanishes on~$c_0$, then
there is an infinite set $A \sub \N$ such that $T$ vanishes on the subspace
$Z_A:=\big\{(x_n)_{n\in \N} \in \ell_\infty: \, x_n=0\text{ for all $n\in \N\setminus A$}\big\}$.
\end{theorem}

Let $K$ be a compact Hausdorff topological space. By Riesz's representation theorem, the dual 
$C(K)^*$ is the Banach space of all real-valued regular Borel measures on~$K$
with the variation norm. The subset of $C(K)^*$ consisting of all regular Borel probability measures on~$K$
is denoted by $P(K)$. Given $\mu\in P(K)$, any $\xi\in C(K)^*$ can be written (in a unique way) as 
$\xi=f \, d\mu+\xi'$ for some $f\in L_1(\mu)$ and some $\xi'\in C(K)^*$ which is singular with respect to~$\mu$;
here $f \, d\mu\in C(K)^*$ is given by $(f\, d\mu)(B)=\int_B f \, d\mu$ for every Borel set $B \sub K$ and, as usual, we write $f=\frac{d\xi}{d\mu}$.

The Banach spaces $\ell_\infty$ and $C(\beta\N)$ are isometrically isomorphic, where $\beta\N$ denotes the Stone-\v Cech compactification of~$\N$ 
with the discrete topology. Recall that $\beta\N$ is the set of all ultrafilters on~$\mathbb{N}$, which is a compact Hausdorff topological space such 
that the family $\widehat{A}:=\{\mathcal{U}\in \beta\N: A\in \mathcal{U}\}$, for $A\sub \N$, forms a basis of clopen sets.
Each $n\in \N$ is identified with the ultrafilter $\{A \sub \N: n\in A\} \in \beta\N$.

We are now ready to prove the main result of this section:

\begin{proof}[Proof of Theorem~\ref{theorem:ellinfty-uncountable-2}]
Let $R:C(\beta\N) \to \ell_\infty$ be the isometric isomorphism satisfying $R(\chi_{\widehat{A}})=\chi_{A}$ for all $A\sub \N$.
Let $\mu_0 \in P(\beta\N)$ be the regular Borel probability measure on~$\beta\N$ satisfying $\mu_0(\{n\})=2^{-n}$ for all $n\in \N$.
Observe that for each $f\in L_1(\mu_0)$ the series of real numbers $\sum_{n\in \N}f(n)2^{-n}$ is absolutely convergent and we have
\begin{equation}\label{eqn:atomic}
	\int_{\widehat{B}}f\, d\mu_0=\sum_{n\in B} f(n)2^{-n}
	\quad
	\text{for every $B \sub \N$.}
\end{equation}
Zorn's lemma ensures the existence of a set $\Delta \sub P(\beta\N)$ containing~$\mu_0$ and consisting of mutually singular elements of~$P(\beta\N)$ 
such that $\Delta$ is maximal (with respect to the inclusion) among all subsets of~$P(\beta\N)$ satisfying those properties.
Then for any $\xi\in C(\beta\N)^*$ we have 
\begin{equation}\label{eqn:l1-sum}
	\xi=\sum_{\mu\in \Delta}\frac{d\xi}{d\mu}\, d\mu, 
\end{equation}
the series being absolutely convergent in $C(\beta\N)^*$, and 
the space $C(\beta\N)^*$ is isometrically isomorphic to the $\ell_1$-sum~$Z$ of the family of Banach spaces $\{L_1(\mu): \mu \in \Delta\}$ via the operator 
$S:C(\beta\N)^* \to Z$ defined by
$$
	\quad S(\xi):=\left(\frac{d\xi}{d\mu}\right)_{\mu\in \Delta} \quad\text{for all $\xi\in C(\beta\N)^*$}
$$
(see, e.g., the proof of \cite[Proposition~4.3.8(iii)]{alb-kal}). 

By Theorem~\ref{theorem:Talagrand} 
we can assume without loss of generality that $W$ is a subspace of~$\ell_\infty^*$.
The conclusion is obvious if $W=\{0\}$, so we assume that $W\neq \{0\}$. 
For each $\mu\in \Delta$, let $\pi_\mu: Z \to L_1(\mu)$ be the canonical projection. 
Since $W$ does not contain subspaces isomorphic to~$\ell_1(\mathfrak{c})$, the same holds for the subspace 
$(S\circ R^*)(W) \sub Z$ and so the set
$$
	\Delta_0:=\big\{\mu\in \Delta: \, \pi_\mu\big((S\circ R^*)(W)\big)\neq \{0\}\big\}
$$
is non-empty and has cardinality $|\Delta_0|<\mathfrak{c}$ (by Lemma~\ref{lemma:Rosenthal-l1}). 

Let $T:\ell_\infty \to \ell_\infty(\Delta_0)$ be the operator defined by
$$
	T(x)(\mu):=\begin{cases}
		\int_{\beta \N} R^{-1}(x) \, d\mu & \text{if $\mu\neq \mu_0$} \\
		0 & \text{if $\mu=\mu_0$}
			\end{cases}
$$
for every $\mu\in \Delta_0$ and for every $x\in \ell_\infty$. Every $\mu \in \Delta \setminus \{\mu_0\}$ is singular with respect to~$\mu_0$ and hence, 
$\mu(\widehat{A})=0$ for every finite set $A \sub \N$. Bearing in mind that 
$c_0=\overline{{\rm span}}(\{\chi_{A}: \text{$A\sub \N$ finite}\}) \sub \ell_\infty$, we deduce that 
$T(x)=0$ for every $x \in c_0$. The fact that $|\Delta_0|<\mathfrak{c}$ allows us to apply 
Theorem~\ref{theorem:Kalton} to get an infinite set $A \sub \N$ such that $T$ vanishes on 
$$
	Z_A:=\big\{(x_n)_{n\in \N} \in \ell_\infty: \, x_n=0\text{ for all $n\in \N\setminus A$}\big\},
$$
that is, 
\begin{equation}\label{eqn:vanishing}
	\int_{\beta \N} R^{-1}(x) \, d\mu=0 
	\quad\text{for every $x\in Z_A$ and for every $\mu \in \Delta_0\setminus \{\mu_0\}$}.
\end{equation}

Define a finitely additive map $\nu:\mathcal{P}(A)\to \ell_\infty$ by $\nu(B):=\chi_B$ for all $B \subseteq A$. 
Note that $\nu$ is not countably additive, because $A$ is infinite and $\|\nu(\{n\})\|=1$ for every $n\in A$. 
To finish the proof we will show that $W$ fails the OT property by checking that
$\varphi \circ \nu\in {\rm ca}(\mathcal{P}(A))$ for arbitrarily $\varphi\in W$. 
Observe first that
$$
	R^*(\varphi)\stackrel{\eqref{eqn:l1-sum}}{=}\sum_{\mu\in \Delta}\frac{dR^*(\varphi)}{d\mu} \, d\mu=
	\sum_{\mu\in \Delta_0}\frac{dR^*(\varphi)}{d\mu} \, d\mu,
$$
the series being absolutely convergent in~$C(\beta\N)^*$. 
Moreover, for each $B \sub A$ we have $\chi_B\in Z_A$ and so \eqref{eqn:vanishing} yields 
$\mu(\widehat{B})=0$ for all $\mu \in \Delta_0\setminus \{\mu_0\}$. Hence,
$$
	(\varphi\circ \nu)(B)=R^*(\varphi)(\chi_{\widehat{B}})=\sum_{\mu\in \Delta_0}\int_{\widehat{B}}\frac{dR^*(\varphi)}{d\mu} \, d\mu=
	\begin{cases}
	0 & \text{if $\mu_0\not\in \Delta_0$} \\
	\int_{\widehat{B}}\frac{dR^*(\varphi)}{d\mu_0} \, d\mu_0 & \text{if $\mu_0\in \Delta_0$.}
	\end{cases}
$$
Therefore, if $\mu_0\not \in \Delta_0$, then $\varphi\circ \nu$ is identically null and so countably additive.
If $\mu_0 \in \Delta_0$, then
$$
	(\varphi\circ \nu)(B)=\int_{\widehat{B}}\frac{dR^*(\varphi)}{d\mu_0} \, d\mu_0\stackrel{\eqref{eqn:atomic}}{=}\sum_{n\in B}
	\frac{dR^*(\varphi)}{d\mu_0}(n)2^{-n}
$$
for every $B \sub A$, where the series $\sum_{n\in A} \frac{dR^*(\varphi)}{d\mu_0}(n)2^{-n}$ is absolutely convergent, 
hence $\varphi \circ \nu$ is countably additive. The proof is finished.
\end{proof}

The converse of Theorem~\ref{theorem:ellinfty-uncountable-2} fails to hold, in general. An example follows:

\begin{example}\label{example:Susumu3}
\rm Let $2\mathbb{N}-1$ (resp., $2\mathbb{N}$) be the set of all odd (resp., even) natural numbers.
With the notation of Theorem~\ref{theorem:Kalton}, let $W \sub \ell_\infty^*$ be the norm-closure of $(Z_{2\mathbb{N}-1})^\perp + \ell_1$.
Then:
\begin{enumerate}
\item[(i)] $W$ is total and contains a subspace isometric to~$\ell_1(2^\mathfrak{c})$; and
\item[(ii)] $W$ fails the OT property.
\end{enumerate} 
Indeed, $W$ is total because it contains~$\ell_1$. Let $P:\ell_\infty \to Z_{2\mathbb{N}}$ be the canonical projection 
and define $\Phi: Z_{2\mathbb{N}}^* \to (Z_{2\mathbb{N}-1})^\perp$ by $\Phi(\xi):=\xi \circ P$ for all $\xi \in Z_{2\mathbb{N}}^*$. Then $\Phi$ is an isometric embedding
and therefore $W$ contains a subspace isometric to~$Z_{2\mathbb N}^*$. Since $Z_{2\mathbb{N}}$ is isometric to $C(\beta\N)$, 
its dual contains a subspace isometric to $\ell_1(|\beta\N|)$, and the same holds for~$W$. Now, bear in mind that
$|\beta \N|=2^\mathfrak{c}$ (see, e.g., \cite[19.13(d)]{wil-J}) to get~(i). In order to show that $W$ fails the OT property, define
$\nu:\mathcal{P}(\N)\to \ell_\infty$ by $\nu(A):=\chi_{A\cap (2\mathbb{N}-1)}=Q(\chi_A)$ for all $A \sub \N$, where $Q:\ell_\infty \to Z_{2\mathbb{N}-1}$
is the canonical projection. Clearly, $\nu$ is not countably additive. However, we claim that $\varphi \circ \nu \in {\rm ca}(\mathcal{P}(\N))$
for every $\varphi\in W$. Indeed, by the Vitali-Hahn-Saks-Nikod\'{y}m theorem (see, e.g., \cite[p.~24, Corollary~10]{die-uhl-J}),
it suffices to check it whenever $\varphi \in (Z_{2\mathbb{N}-1})^\perp \cup \{e_n:n\in \N\}$, where $\{e_n:n\in \N\}$ is the canonical basis of~$\ell_1$. On the one hand, 
we have $\varphi \circ \nu=0$ (hence it is countably additive) whenever $\varphi \in (Z_{2\mathbb{N}-1})^\perp$. On the other hand, 
for each $n\in \N$ the composition $e_n \circ \nu$ is countably additive, because $(e_n\circ \nu)(A)=\chi_{A\cap (2\mathbb{N}-1)}(n)$ for all $A \sub \N$. 
This establishes the claim and hence, $W$ fails the OT property.  \end{example}

\begin{corollary}\label{corollary:smalltotal}
Let $X$ be a Banach space such that there is a subset of~$X^*$ having the OT property but not containing sets equivalent to the canonical 
basis of~$\ell_1(\mathfrak{c})$. Then $X$ does not contain subspaces isomorphic to~$\ell_\infty$.
\end{corollary}
\begin{proof}
Let $W \sub X^*$ be a set having the OT property such that $W$ does not contain sets equivalent to the canonical basis of~$\ell_1(\mathfrak{c})$.
Given any subspace $Y \subseteq X$, the set $W|_Y:=\{x^*|_Y:x^*\in W\} \subseteq Y^*$ has the OT property and does not contain sets equivalent to 
the canonical basis of~$\ell_1(\mathfrak{c})$. By Theorem~\ref{theorem:ellinfty-uncountable-2}, 
$Y$ cannot be isomorphic to~$\ell_\infty$. 
\end{proof}

\begin{remark}
\rm The converse of the previous corollary is not true in general, as witnessed by the space $X=c_0(\mathfrak{c})$.
Indeed, $c_0(\mathfrak{c})$ contains no subspace isomorphic to~$\ell_\infty$ (because 
any separable subspace of $c_0(\mathfrak{c})$ is isomorphic to a subspace of~$c_0$, while
any separable Banach space is isomorphic to a subspace of~$\ell_\infty$). 
The space $c_0(\mathfrak{c})$ is weakly compactly generated, so
its dual $c_0(\mathfrak{c})^*=\ell_1(\mathfrak{c})$ satisfies 
${\rm dens}(\ell_1(\mathfrak{c}),w^*)={\rm dens}(c_0(\mathfrak{c}),\|\cdot \|)=\mathfrak{c}$ (see, e.g., \cite[Theorem~13.3]{fab-ultimo}).
Fix a total set $W \sub \ell_1(\mathfrak{c})$. Then $W$ has the OT property by the Diestel-Faires Theorem~\ref{theorem:DF}(ii).
Since the subspace $W_0:=\overline{{\rm span}(W)}^{\|\cdot\|}$ is $w^*$-dense in $\ell_1(\mathfrak{c})$, 
we have ${\rm dens}(W_0,\|\cdot\|)=\mathfrak{c}$.
Now, a classical result of K\"{o}the (see, e.g., \cite[p.~29]{ros-70}) ensures that $W_0$ contains a subspace isomorphic to~$\ell_1(\mathfrak{c})$. 
Finally, Theorem~\ref{theorem:Talagrand} applies to conclude that $W$ contains a set equivalent to the canonical basis of~$\ell_1(\mathfrak{c})$.
\end{remark}

\begin{corollary}\label{corollary:countableOT}
Let $X$ be a Banach space. The following statements are equivalent:
\begin{enumerate}
\item[(i)] $X$ does not contain subspaces isomorphic to~$\ell_\infty$ and $X^*$ is $w^*$-separable.
\item[(ii)] There is a countable subset of~$X^*$ having the OT property. 
\item[(iii)] There is a norm-separable subset of~$X^*$ having the OT property.
\end{enumerate}
\end{corollary}
\begin{proof}
(i)$\impli$(ii) follows from the  Diestel-Faires Theorem~\ref{theorem:DF}(ii), bearing in mind that the 
$w^*$-separability of~$X^*$ is equivalent to the existence of a countable total subset of~$X^*$.
The implication (ii)$\impli$(i) is a consequence of Corollary~\ref{corollary:smalltotal} and Proposition~\ref{proposition:OTimpliestotal}.
The equivalence (ii)$\Leftrightarrow$(iii) follows from Proposition~\ref{proposition:OT-sequential}.
\end{proof}

\section{$L_1$ spaces of vector measures with separable range}\label{section:L1}

Let $\Sigma$ be a $\sigma$-algebra, let $X$ be a Banach space and let $\nu \in {\rm ca}(\Sigma,X)$. By Theorems~\ref{theorem:BDS} and~\ref{theorem:Rosenthal-separable}, the set 
$$
	\mathcal{R}(\nu)=\{\nu(A): \, A\in \Sigma\} \sub X
$$
(the range of $\nu$) is separable whenever $X=\ell_\infty$ and so, in this case, $\nu$ can be seen as an element of
${\rm ca}(\Sigma,Y)$ for some separable subspace $Y \sub \ell_\infty$. Conversely, 
if $X$ is separable, then it is isometric to a subspace of~$\ell_\infty$
and so $\nu$ can be seen as an element of ${\rm ca}(\Sigma,\ell_\infty)$. As a consequence, we get:

\begin{proposition} \label{proposition:linftyvm}
Let $E$ be a Banach function space over a finite measure space $(\Omega,\Sigma,\mu)$. The following statements are equivalent:
\begin{itemize}
\item[(i)] There is $\nu_0 \in {\rm ca}(\Sigma,\ell_\infty)$ such that $E$ is lattice-isomorphic to~$L_1(\nu_0)$.
\item[(ii)] There exist a separable Banach space $X$ and $\nu_1\in {\rm ca}(\Sigma,X)$ such that $E$ is lattice-isomorphic to~$L_1(\nu_1)$.
\end{itemize}
Moreover, both statements hold if $E$ is order continuous and separable.
\end{proposition}

In general, the separability of $\mathcal{R}(\nu)$ does not imply that the space $L_1(\nu)$ is separable, as witnessed
by the space $L_1(\lambda_I)$ of the usual measure on~$\{0,1\}^I$ for any uncountable set~$I$ (see Subsection~\ref{subsection:Cantor}). 

The following example provides a vector measure $\nu$ such that $\mathcal{R}(\nu)$ is separable,
$\overline{{\rm span}}(\mathcal{R}(\nu))$ is infinite-dimensional and $L_1(\nu)$ is neither separable nor
lattice-isomorphic to any AL-space.

\begin{example}\label{example:SusumuAugust}
\rm Let $G$ be any non-metrizable compact abelian group (e.g., the product $\{0,1\}^I$ or $\mathbb{T}^I$ for any uncountable set~$I$). By
$\mu$ we denote the Haar probability measure on the Borel $\sigma$-algebra $\mathcal{B}(G)$. Fix $g\in L_1(\mu)$ and
define $\nu:\mathcal{B}(G) \to L_\infty(\mu)$ by $\nu(A):=\chi_A * g$ (the convolution) for all $A\in \mathcal{B}(G)$. Then:
\begin{enumerate}
\item[(i)] $\nu \in {\rm ca}(\mathcal{B}(G),L_\infty(\mu))$; 
\item[(ii)] $\mathcal{R}(\nu)$ is relatively norm-compact (hence separable);
\item[(iii)] $L_1(\nu)$ is not separable; and
\item[(iv)] if, in addition, $g\not\in L_\infty(\mu)$, then $\overline{{\rm span}}(\mathcal{R}(\nu))$ is infinite-dimensional and $L_1(\nu)$ is not
lattice-isomorphic to any AL-space.
\end{enumerate}
Indeed, (i) and (iii) follow from parts (I) and (II)(iv) of \cite[Theorem~1]{oka-ric-2} (where $\nu$ is denoted by $m_g^{(\infty)}$).
Statement (ii) was noticed in \cite[Remark~2]{oka-ric-2}. For statement (iv), suppose that $g\not\in L_\infty(\mu)$. Then
$\nu$ has infinite variation, \cite[Theorem~2]{oka-ric-2}, and so $L_1(\nu)$ is not
lattice-isomorphic to any AL-space, \cite[Proposition~2]{cur2}. To prove that
$\overline{{\rm span}}(\mathcal{R}(\nu))$ is infinite-dimensional, assume by way of contradiction that 
$\overline{{\rm span}}(\mathcal{R}(\nu))$ is finite-dimensional. Then the integration operator $I_\nu: L_1(\nu) \to L_\infty(\mu)$ is
absolutely $1$-summing, because its range is contained in $\overline{{\rm span}}(\mathcal{R}(\nu))$.
Therefore, the restriction of $I_\nu$ to $L_\infty(\mu)$, which coincides with the convolution operator $C_g^{(\infty)}:L_\infty(\mu)\to L_\infty(\mu)$
(see \cite[Theorem~1(II)(ii)]{oka-ric-2}), is absolutely $1$-summing as well. This contradicts the fact that $g\not\in L_\infty(\mu)$, 
\cite[Theorem~2]{oka-ric-2}.
\end{example}

This section is devoted to proving the following:

\begin{theorem}\label{theorem:countable-norming-general}
Let $\Sigma$ be a $\sigma$-algebra, let $X$ be a separable Banach space and let $\nu \in {\rm ca}(\Sigma,X)$. 
If $L_1(\nu)^*$ is order continuous, then $L_1(\nu)$ is separable. 
\end{theorem}

We will obtain Theorem~\ref{theorem:countable-norming-general} as a consequence of a more general approach dealing with
the concept of positively norming set introduced in~\cite{san-tra}:

\begin{definition}\label{definition:positivelynorming}
Let $E$ be a Banach lattice. A set $B \sub B_{E^*}\cap (E^*)^+$
is said to be {\em positively norming} if there is a constant $c>0$ such that
$$
	\|x\|_E \leq c \sup_{\varphi \in B}\varphi(|x|) \quad\text{for every $x\in E$}.
$$
\end{definition}

\begin{lemma}\label{lemma:SeparableRangeImpliesCountablePN}
Let $(\Omega,\Sigma)$ be a measurable space, let $X$ be a separable Banach space and let $\nu \in {\rm ca}(\Sigma,X)$. 
Then $L_1(\nu)$ admits a countable positively norming set.
\end{lemma}
\begin{proof}
Let $\mu$ be a Rybakov control measure of~$\nu$, so that $L_1(\nu)$ is a Banach function space over $(\Omega,\Sigma,\mu)$. 
Since $X$ is separable, $B_{X^*}$ is $w^*$-separable and so we can take a $w^*$-dense sequence $(x_n^*)_{n\in \N}$ in~$B_{X^*}$. For each $n\in \N$,
the measure $|x_n^*\circ \nu|$ is $\mu$-continuous and we consider its Radon-Nikod\'{y}m derivative $g_n:=\frac{d|x_n^*\circ \nu|}{d\mu} \in (L_1(\nu))'$ 
and the associated functional $\varphi_{g_n} \in B_{L_1(\nu)^*}\cap (L_1(\nu)^*)^+$ (see Subsections~\ref{subsection:Bfs} and~\ref{subsection:L1spaces}). 
Since $(x_n^*)_{n\in \N}$ is $w^*$-dense in~$B_{X^*}$, for each $f\in L_1(\nu)$ we 
can apply \cite[Lemma~2.2]{rod16} to get
$$
	\|f\|_{L_1(\nu)}
	=\sup_{n\in \N}\, \int_\Omega |f|\, d|x_n^*\circ \nu|
	=\sup_{n\in \N}\, \varphi_{g_n}(|f|),
$$
that is, $\{\varphi_{g_n}:n\in \N\}$ is positively norming.
\end{proof}

It is now clear that Theorem~\ref{theorem:countable-norming-general} will be an immediate consequence of the following:

\begin{theorem}\label{theorem:general}
Let $E$ be a Banach lattice such that both~$E$ and $E^*$ are order continuous. If 
$E$ admits a countable positively norming set, then $E$ is separable. 
\end{theorem}

The proof of Theorem~\ref{theorem:general} requires some previous work. 
Proposition~\ref{proposition:Bfs} below presents a special case when $E$ is a Banach function space and will be used to prove Theorem~\ref{theorem:general}.
 
\begin{lemma}\label{lemma:separable}
Let $E$ be a Banach function space over a finite measure space~$(\Omega,\Sigma,\mu)$. 
\begin{enumerate}
\item[(i)] If $E$ is separable, then $L_1(\mu)$ is separable.
\item[(ii)] If $L_1(\mu)$ is separable and $E$ is order continuous, then $E$ is separable. 
\end{enumerate}
\end{lemma}
\begin{proof}
(i) If $E$ is separable, then so is $S:={\rm span}(\{\chi_A:A\in \Sigma\}) \sub E$. Since the identity map $\iota: E \to L_1(\mu)$ is an operator,
the set $\iota(S)$ is separable. Since $\iota(S)$ is dense in~$L_1(\mu)$, we conclude that $L_1(\mu)$ is separable.

(ii) Since $E$ is order continuous, $S$ is dense in~$E$ (see, e.g., \cite[Remark~2.6]{oka-alt}). So, it suffices
to check that $\{\chi_A:A\in \Sigma\}$ is separable as a subset of~$E$. Now, since $L_1(\mu)$ is separable, there is a sequence $(A_n)_{n\in \N}$
in~$\Sigma$ such that 
$$
	\inf_{n\in \N} \mu(A_n\triangle A)=0 \quad \text{for every $A\in \Sigma$}.
$$
Therefore, the order continuity of~$E$ implies that the set $\{\chi_{A_n}:n\in \N\}$ is dense in $\{\chi_A:A \in \Sigma\}$
as subsets of~$E$ (see, e.g., \cite[Lemma~2.37(ii)]{oka-alt}).
\end{proof}

Given a finite measure space $(\Omega,\Sigma,\mu)$ and $A\in \Sigma \setminus \mathcal{N}(\mu)$, 
we define
$$
	\Sigma_A:=\{B\in \Sigma: \, B \sub A\}
	\quad\mbox{and}\quad
	\mu_A(B):=\mu(A)^{-1} \mu(B) \quad \text{for every $B\in \Sigma_A$,}
$$ 
so that $\mu_A$ is a probability measure on the measurable space $(A,\Sigma_A)$.

\begin{lemma}\label{lemma:Maharam}
Let $(\Omega,\Sigma,\mu)$ be a finite measure space such that $L_1(\mu)$ is not separable 
and let $(g_n)_{n\in \N}$ be a sequence in~$L_1(\mu)$. Then there exist $A\in \Sigma\setminus \mathcal{N}(\mu)$ and 
a sequence $(A_m)_{m\in \N}$ of pairwise disjoint sets of~$\Sigma_A\setminus \mathcal{N}(\mu)$ such that
$$
	\int_{A} f g_n \, d\mu_A=\left(\int_{A} f\, d\mu_A\right) \left(\int_{A} g_n \, d\mu_A\right)
$$
for every $f\in {\rm span}(\{\chi_{A_m}:m\in \N\})$ and for every $n\in \N$.
\end{lemma}
\begin{proof}
 Since $L_1(\mu)$ is not separable, 
Maharam's theorem (see, e.g., \cite[Section~3]{fre14} or \cite[\S14]{lac-J}) ensures the existence of
$A\in \Sigma\setminus \mathcal{N}(\mu)$ such that the measure algebra of~$\mu_A$ is 
isomorphic to the measure algebra of the usual measure $\lambda_I$ on~$\{0,1\}^I$, for some uncountable set~$I$.
Therefore, there is a lattice isometry $\Phi: L_1(\mu_A) \to L_1(\lambda_I)$ satisfying 
\begin{equation}\label{eqn:Xi0}
	\int_A v \, d\mu_A =\int_{\{0,1\}^I} \Phi(v) \, d\lambda_I
	\quad \text{for every $v\in L_1(\mu_A)$}
\end{equation}
and
\begin{equation}\label{eqn:Xi}
	\Phi(uv)=\Phi(u) \Phi(v) 
	\quad \text{whenever $u\in {\rm span}(\{\chi_B:B\in \Sigma_A\})$ and $v\in L_1(\mu_A)$} 
\end{equation}
(see Subsection~\ref{subsection:MeasureAlgebras}). 

We denote by 
$$
	\rho_{J'J}:\{0,1\}^{J'} \to \{0,1\}^{J}
$$ 
the canonical projection for any $J \sub J' \sub I$.

For each $n\in \N$ we have $\Phi(g_n|_A) \in L_1(\lambda_I)$ and so there exist a countable set $I_n \sub I$ and $h_n\in L_1(\lambda_{I_n})$
such that $\Phi(g_n|_A)=h_n\circ \rho_{I I_n}$ (see, e.g., \cite[254Q]{freMT-2}). 
Then the set $I':=\bigcup_{n\in \N}I_n$ is countable and for each $n\in \N$ we have 
\begin{equation}\label{eqn:rhoII}
	\Phi(g_n|_A)=\widetilde{h}_n \circ \rho_{I I'},
\end{equation}
where $\widetilde{h}_n:=h_n\circ \rho_{I' I_n} \in L_1(\lambda_{I'})$. Note that the set $J:=I \setminus I'$ is uncountable and, in particular, infinite. So, we can find a sequence
$(B_m)_{m\in \N}$ of pairwise disjoint elements of~$\Sigma_J\setminus \mathcal{N}({\lambda_J})$. Define 
$$
	C_m:= B_m \times \{0,1\}^{I'} \in \Sigma_I
	\quad \text{for all $m\in \N$}, 
$$
so that
the $C_m$'s are pairwise disjoint with $\lambda_I(C_m)=\lambda_J(B_m)$.

{\em Claim.} For every $h\in {\rm span}(\{\chi_{C_m}:m\in \N\})$ and for every $n\in \N$ we have
\begin{equation}\label{eqn:independence0} 
	\int_{\{0,1\}^I} h\,  \Phi(g_n|_A) \, d\lambda_I=\left(\int_{\{0,1\}^I} h\, d\lambda_I\right) \left(\int_{\{0,1\}^I} \Phi(g_n|_A) \, d\lambda_I\right).
\end{equation}
Indeed, note that $h = \widetilde{h} \circ \rho_{IJ}$ for some $\widetilde{h}\in {\rm span}(\{\chi_{B_m}:m\in \N\})$ and then
Fubini's theorem yields
\begin{multline}\label{eqn:Fubini}
	\int_{\{0,1\}^I} h\,  
	\Phi(g_n|_A) \, d\lambda_I \stackrel{\eqref{eqn:rhoII}}{=} \int_{\{0,1\}^I} (\widetilde{h} \circ \rho_{IJ})\, (\widetilde{h}_n \circ \rho_{I I'}) \, d\lambda_I \\
	=	\left(\int_{\{0,1\}^J} \widetilde{h}\, d\lambda_J\right) \left(\int_{\{0,1\}^{I'}} \widetilde{h}_n \, d\lambda_{I'}\right).
\end{multline}
Since the function $\rho_{I I'}$ is $\Sigma_I$-to-$\Sigma_{I'}$ measurable and 
$\lambda_{I'}(A)=\lambda_I(\rho_{I I'}^{-1}(A))$ for every $A\in \Sigma_{I'}$ (see, e.g., \cite[254O]{freMT-2}), we have
\begin{equation}\label{eqn:cv-1}
	\int_{\{0,1\}^I} \Phi(g_n|_A) \, d\lambda_I\stackrel{\eqref{eqn:rhoII}}{=} \int_{\{0,1\}^I} \widetilde{h}_n \circ \rho_{I I'} \, d\lambda_I =\int_{\{0,1\}^{I'}} \widetilde{h}_n \, d\lambda_{I'}.
\end{equation}
Similarly, or by direct computation, we also have
\begin{equation}\label{eqn:cv-2}
	\int_{\{0,1\}^I} h\, d\lambda_I=\int_{\{0,1\}^J} \widetilde{h}\, d\lambda_J.
\end{equation}
By putting together \eqref{eqn:Fubini}, \eqref{eqn:cv-1} and \eqref{eqn:cv-2} we get~\eqref{eqn:independence0}, as claimed.

Finally, let $(A_m)_{m\in \N}$ be a sequence of pairwise disjoint elements of~$\Sigma_A$ such that 
$\Phi(\chi_{A_m}|_A)=\chi_{C_m}$ for all $m\in \N$.  Then $\mu(A_m)=\mu(A)\lambda_I(C_m)>0$ for all $m\in \N$.
Given any $f\in {\rm span}(\{\chi_{A_m}:m\in \N\})$, we have $\Phi(f|_A)\in {\rm span}(\{\chi_{C_m}:m\in \N\})$  and for each $n\in \N$ we have
\begin{multline*}
	\int_{A} f g_n \, d\mu_A 
	 \stackrel{\eqref{eqn:Xi0}}{=}
	\int_{\{0,1\}^I} \Phi(f|_A \, g_n|_A) \, d\lambda_I 
	 \stackrel{\eqref{eqn:Xi}}{=} 
	\int_{\{0,1\}^I} \Phi(f|_A)\Phi(g_n|_A) \, d\lambda_I \\
	\stackrel{\eqref{eqn:independence0}}{=}
	\left(\int_{\{0,1\}^I} \Phi(f|_A)\, d\lambda_I\right) \left(\int_{\{0,1\}^I} \Phi(g_n|_A) \, d\lambda_I\right) \stackrel{\eqref{eqn:Xi0}}{=}
	 \left(\int_{A} f\, d\mu_A\right) \left(\int_{A} g_n \, d\mu_A\right).
\end{multline*}
The proof is finished.
\end{proof}

We shall make use of the following well known characterization of Banach lattices with order continuous dual (see, e.g., \cite[Theorem~4.69]{ali-bur}):

\begin{theorem}\label{theorem:AB}
Let $E$ be a Banach lattice. The following statements are equivalent:
\begin{enumerate}
\item[(i)] $E^*$ is order continuous.
\item[(ii)] There is no disjoint sequence in~$E^+$ which is equivalent to the canonical basis of~$\ell_1$.
\item[(iii)] $E$ does not contain sublattices which are lattice isomorphic to~$\ell_1$.
\item[(iv)] $E^*$ does not contain subspaces isomorphic to~$\ell_\infty$.
\end{enumerate}
\end{theorem}

\begin{proposition}\label{proposition:Bfs}
Let $E$ be a Banach function space over a finite measure space $(\Omega,\Sigma,\mu)$ such that both $E$ and $E^*$ are order continuous. If 
$E$ admits a countable positively norming set, then $E$ is separable. 
\end{proposition}
\begin{proof} 
Fix a sequence $(\phi_n)_{n\in \N}$ in $B_{E^*}\cap (E^{*})^+$ and a constant $c>0$ such that
\begin{equation}\label{eqn:positively-norming0}
	\|f\|_{E} \leq c \sup_{n\in \N}\, \phi_n(|f|)
	\quad
	\text{for every $f\in E$}. 
\end{equation}
Since $E$ is order continuous, we can identify $E^*$ with~$E'$, hence
for each $n\in \N$ we have $\phi_n=\varphi_{g_n}$ for some $g_n\in E'$
(see Subsection~\ref{subsection:Bfs}). Then \eqref{eqn:positively-norming0} reads as
\begin{equation}\label{eqn:positively-norming}
	\|f\|_{E} \leq c \sup_{n\in \N}\, \int_\Omega |f|g_n \, d\mu
	\quad
	\text{for every $f\in E$}. 
\end{equation}

Suppose, by contradiction, that $E$ is not separable. By Lemma~\ref{lemma:separable}(ii), the space $L_1(\mu)$ is not separable. 
Then we can apply Lemma~\ref{lemma:Maharam} to $(g_n)_{n\in \N}$ (as a sequence in~$L_1(\mu)$) to find $A\in \Sigma\setminus \mathcal{N}(\mu)$ and 
a sequence $(A_m)_{m\in \N}$ of pairwise disjoint sets of~$\Sigma_A\setminus \mathcal{N}(\mu)$ such that
\begin{equation}\label{eqn:independence}
	\int_{A} f g_n \, d\mu_A=\left(\int_{A} f\, d\mu_A\right) \left(\int_{A} g_n \, d\mu_A\right)
\end{equation}
for every $f\in {\rm span}(\{\chi_{A_m}:m\in \N\})$ and for every $n\in \N$.

It follows that for each $f\in {\rm span}(\{\chi_{A_m}:m\in \N\})$ we have
\begin{eqnarray*}
	(c\mu(A))^{-1} \|f\|_{E} & \stackrel{\eqref{eqn:positively-norming}}{\leq} & \sup_{n\in \N}\int_A |f| g_n \, d\mu_A 
	\\
	& \stackrel{\eqref{eqn:independence}}{=} & \sup_{n\in \N}\left(\int_{A} |f|\, d\mu_A\right) \left(\int_{A} g_n \, d\mu_A\right) \\ 
	& \leq & 
	\mu(A)^{-2} \, \|\chi_A\|_E \, \|f\|_{L_1(\mu)}   \\
	& \leq & \mu(A)^{-2} \|\chi_A\|_E \, \|\iota\| \, \|f\|_E,
\end{eqnarray*}
where $\iota:E \to L_1(\mu)$ is the identity operator. 

Therefore, there exist constants $\alpha,\beta>0$ such that 
\begin{equation}\label{eqn:equivalence}
	\alpha \|f\|_{L_1(\mu)} \leq \|f\|_{E} \leq \beta\|f\|_{L_1(\mu)} 
	\quad
	\text{for every $f\in {\rm span}(\{\chi_{A_m}:m\in \N\})$}. 
\end{equation}

Define $f_m:=\mu(A_m)^{-1}\chi_{A_m}\in E$ for all $m\in \N$. 
We can use~\eqref{eqn:equivalence} to prove that 
the disjoint sequence $(f_m)_{m\in \N}$ in~$E^+$ is equivalent to the canonical basis of~$\ell_1$.
This contradicts the order continuity of~$E^*$ (see Theorem~\ref{theorem:AB}).
\end{proof}

Recall that an {\em unconditional Schauder decomposition} of a Banach space~$X$ is a family~$\{X_i:i\in I\}$ 
of subspaces of~$X$ such that each $x\in X$ can be written in a unique way as $x=\sum_{i\in I} x_i$, where $x_i\in X_i$ for all $i\in I$,
the series being unconditionally convergent. In this case, for each $i\in I$ one has a projection $P_i$ from~$X$ onto~$X_i$
in such a way that $x=\sum_{i\in I} P_i(x)$ for all $x\in X$. 

\begin{lemma}\label{lemma:DF}
Let $X$ be a Banach space and let $\{X_i: i\in I\}$ be an unconditional Schauder decomposition of~$X$.
For each $i\in I$, let $P_i$ be the associated projection from~$X$ onto~$X_i$.
If $X^*$ contains no subspace isomorphic to~$\ell_\infty$,
then for every $x^* \in X^*$ the set $\{i\in I: x^*\circ P_i\neq 0\}$ is countable.
\end{lemma}
\begin{proof}
For each $J \sub I$, let $Q_J$ be the projection from~$X$ onto $\overline{{\rm span}}(\bigcup_{i\in J}X_i)$ defined by $Q_J(x):=\sum_{i\in J} P_i(x)$ for all $x\in X$.

Fix $x^*\in X^*$ and define $\nu: \mathcal{P}(I) \to X^*$ by
$$
	\nu(J):=x^* \circ Q_J \quad
	\text{for all $J \sub I$}.
$$
We identify $X$ as a subspace of~$X^{**}$ in the canonical way.
Note that $x \circ \nu \in {\rm ca}(\mathcal{P}(I))$ for every $x\in X$, because
the series of real numbers $\sum_{i\in I}x^*(P_i(x))$ is absolutely convergent and
$$
	(x \circ \nu)(J)=x^*(Q_J(x))=\sum_{i\in J}x^*(P_i(x))
	\quad \text{for all $J \sub I$}.
$$
Since $X^*$ contains no subspace isomorphic to~$\ell_\infty$ and $X$ is a total subset of~$X^{**}$, we can apply the Diestel-Faires Theorem~\ref{theorem:DF}(ii)
to conclude that $\nu \in {\rm ca}(\mathcal{P}(I),X^*)$.

We claim that for every $\epsilon>0$ the set $I_\epsilon:=\{i\in I: \|x^* \circ P_i\|\geq \epsilon\}$ is finite. Indeed, 
if not, then there is a sequence $(i_n)_{n\in \N}$ of distinct elements of~$I_\epsilon$. However, the countable additivity of~$\nu$ 
implies that the series $\sum_{n\in \N}\nu(\{i_n\})=\sum_{n\in \N}x^* \circ P_{i_n}$ is unconditionally convergent in~$X^*$, which is impossible 
because $\|x^* \circ P_{i_n}\|\geq \epsilon$ for all $n\in \N$. Therefore, the set $\{i\in I: x^*\circ P_i\neq 0\}=\bigcup_{n\in \N}I_{1/n}$ is countable.
\end{proof}

\begin{lemma}\label{lemma:spf}
Let $E$ be a Banach lattice admitting a countable positively norming set. Then $E$ admits a strictly positive functional, i.e., 
there is $\varphi \in (E^*)^+$ such that $\varphi(x)>0$ whenever $x\in E^+\setminus \{0\}$.
\end{lemma}
\begin{proof}
Take a sequence $(\varphi_n)_{n\in \N}$ in $B_{E^*}\cap (E^{*})^+$ and a constant $c>0$ such that
$\|x\|_{E}\leq c\sup_{n\in \N}\varphi_n(|x|)$ for every $x\in E$. Now, it is clear that the functional
$\varphi:=\sum_{n\in \N}2^{-n}\varphi_n$ satisfies the required property.
\end{proof}

We have gathered all the tools needed to prove the main result of this section:

\begin{proof}[Proof of Theorem~\ref{theorem:general}]
Since $E$ is order continuous, it admits an unconditional Schauder decomposition $\{E_i:i\in I\}$ consisting 
of pairwise disjoint bands, each having a weak order unit (see, e.g., \cite[Proposition 1.a.9]{lin-tza-2}).
For each $i\in I$, let $P_i$ be the associated projection from~$E$ onto~$E_i$.

Fix $i\in I$. Since $E_i$ is order continuous and has a weak order unit, it is 
lattice-isometric to a Banach function space over a finite measure space (see, e.g., \cite[Proposition 1.b.14]{lin-tza-2}).
Since $E^*$ is order continuous, the same holds for~$E_i^*$ (bear in mind the equivalence (i)$\Leftrightarrow$(iii) in Theorem~\ref{theorem:AB}).
Moreover, $E_i$ admits a countable positively norming set (consider the restriction to~$E_i$ of a countable positively norming set for~$E$).
Then, $E_i$ is separable by Proposition~\ref{proposition:Bfs}.

Fix $\varphi \in (E^*)^+$ such that $\varphi(x)>0$ whenever $x\in E^+\setminus \{0\}$ (see Lemma~\ref{lemma:spf}).
For each $i\in I$ we have $E_i\neq \{0\}$ and so $\varphi\circ P_i\neq 0$.
Since $E^*$ is order continuous, it contains no subspace isomorphic to~$\ell_\infty$
(see Theorem~\ref{theorem:AB}) and then Lemma~\ref{lemma:DF} implies
that $I$ is countable. From the separability of each~$E_i$ it follows that $E$ is separable.
\end{proof}

\begin{example}\label{example:no-oc-dual}
\rm The conclusion of Theorem~\ref{theorem:general} can fail if the order continuity of~$E$ is dropped. For instance,
the non-separable Banach lattice $\ell_\infty$ admits a countable positively norming set (the coordinate functionals form a norming set)
and $\ell_\infty^*$ is an AL-space, hence it is order continuous (see, e.g., \cite[p.~194 and Theorem~4.23]{ali-bur}).
\end{example}

Any reflexive Banach lattice is order continuous (see, e.g., \cite[Theorem~4.9]{ali-bur}), so we get:

\begin{corollary}\label{corollary:reflexive}
Let $E$ be a reflexive Banach lattice. If $E$ admits a countable positively norming set, then $E$ is separable. 
\end{corollary}

The previous corollary and Lemma~\ref{lemma:SeparableRangeImpliesCountablePN} yield:

\begin{corollary}\label{corollary:reflexiveL1}
Let $\Sigma$ be a $\sigma$-algebra, let $X$ be a separable Banach space and let $\nu \in {\rm ca}(\Sigma,X)$. 
If $L_1(\nu)$ is reflexive, then $L_1(\nu)$ is separable. 
\end{corollary}
 
\begin{example}\label{example:Lp}
\rm Let $(\Omega,\Sigma,\mu)$ be a finite measure space such that $L_1(\mu)$ is not separable and let $1<p<\infty$. Then 
\begin{enumerate}
\item[(i)] $L_p(\mu)$ is a non-separable reflexive Banach function space over $(\Omega,\Sigma,\mu)$.
\item[(ii)] $L_p(\mu)=L_1(\nu)$, where
$\nu\in {\rm ca}(\Sigma,L_p(\mu))$ is defined by $\nu(A):=\chi_A$ for all $A\in \Sigma$. 
\item[(iii)] If $X$ is a {\em separable} Banach space, $\Sigma_0$ is a $\sigma$-algebra and $\nu_0 \in {\rm ca}(\Sigma_0,X)$, then
$L_p(\mu)$ is not isomorphic to $L_1(\nu_0)$ even as Banach spaces. To see this, apply Corollary~\ref{corollary:reflexiveL1} and part~(i) above.
\end{enumerate}
\end{example}

\subsection*{Acknowledgements}
We thank A. Avil\'{e}s for valuable comments on Theorem~\ref{theorem:ellinfty-uncountable-2}.
The research of J. Rodr\'{i}guez was partially supported by grants MTM2017-86182-P 
(funded by MCIN/AEI/10.13039/501100011033 and ``ERDF A way of making Europe'') and 
21955/PI/22 (funded by {\em Fundaci\'on S\'eneca - ACyT Regi\'{o}n de Murcia}). 
The research of E.A. S\'{a}nchez-P\'{e}rez was partially supported 
by grant PID2020-112759GB-I00 funded by MCIN/AEI/10.13039/501100011033.

%\bibliography{C:/Mat/Bibliografia/inves,C:/Mat/Bibliografia/invesJose}

\def\cprime{$'$}\def\cdprime{$''$}
  \def\polhk#1{\setbox0=\hbox{#1}{\ooalign{\hidewidth
  \lower1.5ex\hbox{`}\hidewidth\crcr\unhbox0}}} \def\cprime{$'$}
\providecommand{\bysame}{\leavevmode\hbox to3em{\hrulefill}\thinspace}
\providecommand{\MR}{\relax\ifhmode\unskip\space\fi MR }
\providecommand{\MRhref}[2]{%
  \href{http://www.ams.org/mathscinet-getitem?mr=#1}{#2}
}
\providecommand{\href}[2]{#2}

\bibliographystyle{amsplain}

\end{document}